\documentclass{bmcart}

\usepackage[utf8]{inputenc} %unicode support
\usepackage{graphicx}	
\usepackage{color}
\usepackage{natbib}

\usepackage{amsfonts,amsmath,latexsym,amssymb} % these packages are required
\usepackage{amsthm}                % please use one of this two options for theorems
\usepackage{bm}
\usepackage{hyperref}

\newtheorem{theorem}{Theorem}           
\newtheorem{lemma}{Lemma}               
\newtheorem{remark}{Remark}               
\newtheorem{corollary}{Corollary}
\newtheorem{proposition}{Proposition}

\startlocaldefs

\newcommand{\E}{\mathbb{E}}

\newcommand{\edr}{\mathrm{e}}
\newcommand{\ddr}{\mathrm{d}}

\endlocaldefs

\begin{document}

%%% Start of article front matter
\begin{frontmatter}

\begin{fmbox}
\dochead{Research}


\title{On the Hurwitz zeta function with an application to the beta-exponential distribution}

%%%%%%%%%%%%%%%%%%%%%%%%%%%%%%%%%%%%%%%%%%%%%%
%%                                          %%
%% Enter the authors here                   %%
%%                                          %%
%% Specify information, if available,       %%
%% in the form:                             %%
%%   <key>={<id1>,<id2>}                    %%
%%   <key>=                                 %%
%% Comment or delete the keys which are     %%
%% not used. Repeat \author command as much %%
%% as required.                             %%
%%                                          %%
%%%%%%%%%%%%%%%%%%%%%%%%%%%%%%%%%%%%%%%%%%%%%%

\author[
   addressref={aff1},                   % id's of addresses, e.g. {aff1,aff2}
   corref={aff1},                       % id of corresponding address, if any
%   noteref={n1},                        % id's of article notes, if any
   email={julyan.arbel@inria.fr}   % email address
]{\inits{J}\fnm{Julyan} \snm{Arbel}}
\author[
   addressref={aff2},
   email={olivier.marchal@univ-st-etienne.fr}
]{\inits{O}\fnm{Olivier} \snm{Marchal}}
\author[
   addressref={aff3},
   email={bernardo.nipoti@unimib.it}
]{\inits{B}\fnm{Bernardo} \snm{Nipoti}}

%%%%%%%%%%%%%%%%%%%%%%%%%%%%%%%%%%%%%%%%%%%%%%
%%                                          %%
%% Enter the authors' addresses here        %%
%%                                          %%
%% Repeat \address commands as much as      %%
%% required.                                %%
%%                                          %%
%%%%%%%%%%%%%%%%%%%%%%%%%%%%%%%%%%%%%%%%%%%%%%

\address[id=aff1]{%                           % unique id
  \orgname{Universit\'e Grenoble Alpes, Inria, CNRS, Grenoble INP, LJK}, % university, etc
%   \street{Waterloo Road},                     %
  \postcode{38000}                                % post or zip code
  \city{Grenoble},                              % city
  \cny{France}                                    % country
}
\address[id=aff2]{%
  \orgname{Universit\'e de Lyon, CNRS UMR 5208, Universit\'e Jean Monnet, Institut Camille Jordan},
%   \street{D\"{u}sternbrooker Weg 20},
  \postcode{69000}
  \city{Lyon},
  \cny{France}
}
\address[id=aff3]{%
  \orgname{Department of Economics, Management and Statistics, 
Universit\`a degli Studi di Milano Bicocca},
%   \street{D\"{u}sternbrooker Weg 20},
%   \postcode{69000}
  \city{Milan},
  \cny{Italy}
}

%%%%%%%%%%%%%%%%%%%%%%%%%%%%%%%%%%%%%%%%%%%%%%
%%                                          %%
%% Enter short notes here                   %%
%%                                          %%
%% Short notes will be after addresses      %%
%% on first page.                           %%
%%                                          %%
%%%%%%%%%%%%%%%%%%%%%%%%%%%%%%%%%%%%%%%%%%%%%%

% \begin{artnotes}
% %\note{Sample of title note}     % note to the article
% \note[id=n1]{Equal contributor} % note, connected to author
% \end{artnotes}

\end{fmbox}% comment this for two column layout

%%%%%%%%%%%%%%%%%%%%%%%%%%%%%%%%%%%%%%%%%%%%%%
%%                                          %%
%% The Abstract begins here                 %%
%%                                          %%
%% Please refer to the Instructions for     %%
%% authors on http://www.biomedcentral.com  %%
%% and include the section headings         %%
%% accordingly for your article type.       %%
%%                                          %%
%%%%%%%%%%%%%%%%%%%%%%%%%%%%%%%%%%%%%%%%%%%%%%

\begin{abstractbox}

\begin{abstract} % abstract
        We prove a monotonicity property of the Hurwitz zeta function which, in turn, translates into a chain of inequalities for polygamma functions of different orders. We provide a probabilistic interpretation of our result by exploiting a connection between Hurwitz zeta function and the cumulants of the beta-exponential distribution.
\end{abstract}

%%%%%%%%%%%%%%%%%%%%%%%%%%%%%%%%%%%%%%%%%%%%%%
%%                                          %%
%% The keywords begin here                  %%
%%                                          %%
%% Put each keyword in separate \kwd{}.     %%
%%                                          %%
%%%%%%%%%%%%%%%%%%%%%%%%%%%%%%%%%%%%%%%%%%%%%%


% MSC classifications codes, if any
%\begin{keyword}[class=AMS]
%\kwd[Primary ]{}
%\kwd{}
%\kwd[; secondary ]{}
%\end{keyword}

\end{abstractbox}
%
%\end{fmbox}% uncomment this for twcolumn layout

\end{frontmatter}

%%%%%%%%%%%%%%%%%%%%%%%%%%%%%%%%%%%%%%%%%%%%%%
%%                                          %%
%% The Main Body begins here                %%
%%                                          %%
%% Please refer to the instructions for     %%
%% authors on:                              %%
%% http://www.biomedcentral.com/info/authors%%
%% and include the section headings         %%
%% accordingly for your article type.       %%
%%                                          %%
%% See the Results and Discussion section   %%
%% for details on how to create sub-sections%%
%%                                          %%
%% use \cite{...} to cite references        %%
%%  \cite{koon} and                         %%
%%  \cite{oreg,khar,zvai,xjon,schn,pond}    %%
%%  \nocite{smith,marg,hunn,advi,koha,mouse}%%
%%                                          %%
%%%%%%%%%%%%%%%%%%%%%%%%%%%%%%%%%%%%%%%%%%%%%%

%%%%%%%%%%%%%%%%%%%%%%%%% start of article main body
% <put your article body there>

%%%%%%%%%%%%%%%%
%% Background %%
%%

\section{Main result}\label{sec:main}

Let $\zeta(x,s)=\underset{k=0}{\overset{+\infty}{\sum}} (k+s)^{-x}$ be the Hurwitz zeta function \citep{berndt1972hurwitz,srivastava2012zeta} defined for $(x,s)\in (1,+\infty)\times(0,+\infty)$, and, for any $a>0$ and $b>0$, consider the function
\begin{equation}\label{eq:main}
    x\mapsto f(x,a,b) = \left(
    \zeta(x,b) - \zeta(x,a+b)
    \right)^{\frac{1}{x}}, %\quad \text{over} \quad (1,\infty)\times (0,\infty)\times (0,\infty).
\end{equation}
defined on $[1,+\infty)$, where $f(1,a,b)$ is defined by continuity as
\begin{equation}\label{eq:main2}
    f(1,a,b)=\sum_{k=0}^\infty\left( \frac{1}{k+b}-\frac{1}{k+a+b}\right)=\sum_{k=0}^\infty \frac{a}{(k+b)(k+a+b)}.
\end{equation}
The function $f(x,a,b)$ can be alternatively written, with a geometric flavour, as $$f(x,a,b) = \left(\Vert \bm{v}_{a+b} \Vert_x^x - \Vert \bm{v}_{b} \Vert_x^x \right)^{\frac{1}{x}},$$
where, for any $s>0$, $\bm{v}_s$ is an infinite-dimensional vector whose $k^{\text{th}}$ component coincides with $(k-1+s)^{-1}$.\\

% defined as
% $$\bm{v}_s = \left(\frac{1}{k+s}\right)_{k\in\mathbb{N}},$$
% where $\mathbb{N}$ denotes the set of natural integers.\\

% Then $f$ can be rewritten as

% $$f(x,a,b) = \left(\Vert \bm{v}_{a+b} \Vert_x^x - \Vert \bm{v}_{b} \Vert_x^x \right)^{1/x}.$$

The main result of the paper establishes that the function $x\mapsto f(x,a,b)$ is monotone on $[1,+\infty)$ with variations only determined by the value of $a$. More specifically,

\begin{theorem}\label{thm:main}
        For any $b>0$, the function $x\mapsto f(x,a,b)$ defined on $[1,+\infty)$ is increasing\footnote{Throughout the paper, we say that a function $f$ is increasing (resp. decreasing) if $x<y$ implies $f(x)<f(y)$ (resp. $f(x)>f(y)$), and that a quantity $A$ is positive (resp. negative) if $A>0$ (resp. $A<0$).} if $0<a<1$, decreasing if $a>1$, and constantly equal to $\frac{1}{b}$ if $a=1$.
\end{theorem}

Remarkably, when the first argument $x$ of $f$ is a positive integer, say  $x=n\in\mathbb{N}\setminus\{0\}$, the monotonicity property established by Theorem \ref{thm:main} translates into a chain of inequalities in terms of polygamma functions of different orders, which might be of independent interest. Namely, for any $b>0$ and any $0<a_1<1<a_2$, the following holds:
\begin{equation}\label{eq:poly}
    \begin{cases}\psi^{(0)}(b+a_1)-\psi^{(0)}(b)<\ldots<
\left(\frac{\psi^{(n)}(b+a_1)-\psi^{(n)}(b)}{n!}\right)^{\frac{1}{n+1}}<\ldots <
    \frac{1}{b},\\
    \psi^{(0)}(b+a_2)-\psi^{(0)}(b)>\ldots>
\left(\frac{\psi^{(n)}(b+a_2)-\psi^{(n)}(b)}{n!}\right)^{\frac{1}{n+1}}>\ldots >
    \frac{1}{b},
%     &<\ldots<\left(\frac{\psi^{(n)}(b+a_2)-\psi^{(n)}(b)}{(n-1)!}\right)^{1/n}
%   <\ldots<
%     \psi^{(0)}(b+a_2)-\psi^{(0)}(b).
    \end{cases}
\end{equation}
where $\psi^{(m)}$ for $m\in\mathbb{N}$ denotes the polygamma function of order $m$, defined as the derivative of order $m+1$ of the logarithm of the gamma function.
%Inequalities in \eqref{eq:poly} are obtained by using the identities $\psi^{(0)}(s)=-\gamma+\underset{k=0}{\overset{\infty}{\sum}} \frac{s-1}{(k+1)(k+s)}$, where $\gamma$ is the Euler-Mascheroni constant, , which holds for any $s>-1$ \citep[see Identity 6.3.16 in][]{Abr65}, and $\psi^{(n-1)}(s)=(-1)^{n} (n-1)! \zeta(n,s)$ when $n>1$. %and for $n>1$ $\psi^{(0)}(s)=-\gamma+\underset{k=0}{\overset{\infty}{\sum}} \frac{s-1}{(k+1)(k+s)}$, where $\gamma$ is the Euler-Mascheroni constant, which holds for any $s>-1$ \citep[see Identity 6.3.16 in][]{Abr65}.

Theorem~\ref{thm:main} and the derived inequalities in \eqref{eq:poly} add to the current body of literature about inequalities and monotonicity properties of the Hurwitz zeta function \citep{berndt1972hurwitz,simsek2006q,simsek2007twisted,srivastava2011two,leping2013hilbert} and polygamma functions \citep{alzer1998inequalities,alzer2001mean,batir2005some,qi2010complete,guo2015sharp}, respectively. The last part of the statement of Theorem~\ref{thm:main} is immediately verified as,
when $a=1$, $f(x,a,b)$ simplifies to a telescoping series which gives $f(x,1,b)=\frac{1}{b}$ for every $x\in[1,+\infty)$. The rest of the proof is presented in Section~\ref{sec:proof} while  Section~\ref{sec:application} is dedicated to an application of Theorem~\ref{thm:main} to the study of the so-called beta-exponential distribution \citep{gupta1999theory,nadarajah2006beta}, obtained by applying a log-transformation to a beta distributed random variable. Specifically, the chain of inequalities in \eqref{eq:poly}, formally derived in Section~\ref{sec:application}, nicely translates into an analogous monotonicity property involving the cumulants of the beta-exponential distribution. Furthermore, the dichotomy observed in Theorem~\ref{thm:main}, determined by the position of $a$ with respect to 1, is shown to hold for the beta-exponential distribution at the level of (i) its cumulants (whether function~\eqref{eq:cumulant-eb} is increasing or not), (ii) its dispersion (Corollary~\ref{cor:EB-dispersion}), (iii) the shape of its density (log-convex or log-concave, Proposition~\ref{prop:shape-density} and Figure~\ref{fig:proba-interpretation})  and (iv) its hazard function (increasing or decreasing, Proposition~\ref{prop:hazard}).

%\textcolor{red}{Insert some discussion on the case when $x$ in Theorem 1 is an integer, for Referee 2 Point 3. Note that \cite{simsek2006q} reference is already included in the list above, with Hurwitz zeta function properties.}

% \begin{figure}
%     \centering
%     \includegraphics[width = .47\textwidth]{Figures/densities-1.pdf}
%     \includegraphics[width = .47\textwidth]{Figures/f-functions-1.pdf}
%     \caption{Left: density function of the %beta-exponential
%     EB$(a,1)$ distribution for values of $a$ $[0.4,4]$: densities for $0<a<1$ are log-convex (in blue), for $a>1$ are log-concave (in red), while $a=1$ corresponds to the exponential distribution (in black).\\
%     Right: Illustration of Theorem~\ref{thm:main} and Corollary~\ref{cor:EB} representing curves $b\mapsto f_{\text{BE}}(n,a,b)$. Parameter $a$ takes values 0.1 (blue curves), 1 (black curve) and 10 (red curves). Parameter $n$ takes values 1 (continuous curve), 2 (dashed curve) and 3 (dotted curve).}
%     \label{fig:proba-interpretation}
% \end{figure}

\section{Proof of Theorem~\ref{thm:main}\label{sec:proof}}

The proof of Theorem~\ref{thm:main} relies on Lemma~\ref{lem:a-larger-1}, stated below. Lemma~\ref{lem:a-larger-1} considers two sequences and establishes the monotonicity of a third one, function of the first two, whose direction depends on how the two original sequences compare with each other. The same dichotomy, in Theorem~\ref{thm:main}, is driven by the position of the real number $a$ with respect to 1.

%The proof of Theorem~\ref{thm:main} relies on Lemma~\ref{lem:a-larger-1}, stated below, which will be used to compare sequences. %in the proof of the main result. 
%More specifically,  Lemma~\ref{lem:a-larger-1} %establishes inequalities in two different directions, depending on how the aforementioned sequences compare with each another. In Theorem~\ref{thm:main}, this dichotomy is driven by the position of the real number $a$  with respect to 1.

\subsection{A lemma on the monotonicity and the limit of some sequence}
\begin{lemma}\label{lem:a-larger-1}Let $(s_n)_{n\geq 1}$ and $(r_n)_{n\geq 1}$ be two sequences in $(0,1)$ and define, for $N\geq 1$,
\begin{equation*}
	 u_N\overset{\mathrm{def}}{=}\left(1+\sum_{n=1}^N (s_n-r_n)\right)\ln\left(1+\sum_{n=1}^N (s_n-r_n)\right)-\sum_{n=1}^N(s_n\ln s_n-r_n\ln r_n).
\end{equation*}
We define by convention $u_0=0$.
Then two cases are considered:
\begin{enumerate}
	\item[\textbf{1.}] if, for any $n\geq 1$, $r_n\leq s_n$ then,
	for all $N\geq 0$, we have $u_{N+1}\geq u_N$, with the equality holding if and only if $s_{N+1}=r_{N+1}$;
	\item[\textbf{2.}] if, for any $n\geq 1$, $s_{n+1}\leq r_{n+1}\leq s_n\leq r_n$ then, for all $N\geq 0$, we have $u_{N+1}\leq u_N$, with the equality holding if and only if $s_{N+1}=r_{N+1}$.
\end{enumerate}
Moreover, if $\underset{n=1}{\overset{\infty}{\sum}} |s_n-r_n|<\infty$ (implying absolute convergence of the series $\underset{n=1}{\overset{\infty}{\sum}} (s_n\ln s_n-r_n\ln r_n)$) then 
\begin{equation*}
	u_\infty \overset{\mathrm{def}}{=}\lim_{N\to +\infty}u_N= \left(1+\sum_{n=1}^{\infty} (s_n-r_n)\right)\ln\left(1+\sum_{n=1}^{\infty} (s_n-r_n)\right)-\sum_{n=1}^{\infty}(s_n\ln s_n-r_n\ln r_n)
\end{equation*}
exists and satisfies $u_\infty\geq 0$ in case \textbf{1}, while $u_\infty\leq 0$ in case \textbf{2}. In both cases, $u_\infty= 0$ if and only if the two sequences $(r_n)_{n\geq 1}$ and $(s_n)_{n\geq 1}$ equal each other.
\end{lemma}

\begin{remark}\label{rem1}
Note that, in case \textbf{2}, we have 
\begin{align*}
	1+\sum_{n=1}^N (s_n-r_n)=(1-r_1)+s_{N}+\sum_{n=1}^{{N-1}}(s_n-r_{n+1})\geq (1-r_1)+s_{N}> 0,
\end{align*}
\sloppy{so that all quantities defined in the lemma make sense. The absolute convergence of $\underset{n=1}{\overset{\infty}{\sum}} (s_n\ln s_n-r_n\ln r_n)$, stated in Lemma~\ref{lem:a-larger-1}, follows directly from the trivial inequalities}
\begin{align*}
	0\leq s\ln s -r\ln r\leq s-r\,\,,\,\,\forall\,\, 0<r\leq s<1.
\end{align*}
\end{remark}

\smallskip

\begin{flushleft}
\textbf{Proof of Lemma~\ref{lem:a-larger-1}}
\end{flushleft}

\medskip

\textbf{Proof for $N=0$}.
We first study the case $N=0$ and define
\begin{align*}
	h_{r_1}(s_1)=(1+s_1-r_1)\ln(1+s_1-r_1)-s_1\ln s_1+r_1\ln r_1.
\end{align*}
For $s_1=r_1$ we trivially have $h_{r_1}(r_1)=0$. A straightforward computation shows that
\begin{align*}
	h_{r_1}'(s_1)=\ln(1+s_1-r_1)-\ln s_1=\ln((1-r_1)+s_1)-\ln s_1>0,
\end{align*}
since $r_1<1$. Hence $h_{r_1}$ is an increasing function on $(0,1)$. Since $h_{r_1}(r_1)=0$, we immediately get that $h_{r_1}$ is positive on $(r_1,1)$ and negative on $(0,r_1)$, thus proving both cases for $N=0$.

\medskip

\textbf{Proof for $N\geq 1$}.
We now consider the case $N\geq 1$ and define
\begin{align*}
	h_{r_1,\ldots,r_{N+1},s_1,\ldots,s_N}(s_{N+1})=& \, u_{N+1}-u_N\\
	=&\, \left(1+\sum_{n=1}^{N+1} (s_n-r_n)\right)\ln\left(1+\sum_{n=1}^{N+1} (s_n-r_n)\right)\\
	&-\left(1+\sum_{n=1}^{N} (s_n-r_n)\right)\ln\left(1+\sum_{n=1}^{N} (s_n-r_n)\right)\\
	&-s_{N+1}\ln s_{N+1}+r_{N+1}\ln r_{N+1}.
\end{align*}
We trivially get that $h_{r_1,\ldots,r_{N+1},s_1,\ldots,s_N}(r_{N+1})=0$. Moreover, we have
\begin{align*}
	h_{r_1,\ldots,r_{N+1},s_1,\ldots,s_N}'(s_{N+1})=&\ln\left(1+\sum_{n=1}^{N+1} (s_n-r_n)\right)-\ln s_{N+1},\\
	\overset{\text{(1)}}{=}&\,\,\ln\left(s_{N+1}+(1-r_{N+1})+\sum_{n=1}^{N} (s_n-r_n)\right)-\ln s_{N+1},\\
	\overset{\text{(2)}}{=}&\,\,\ln\left(s_{N+1}+(1-r_{1})+\sum_{n=1}^{N} (s_n-r_{n+1})\right)-\ln s_{N+1}.
\end{align*}
Equality $(1)$ shows that $h_{r_1,\ldots,r_{N+1},s_1,\ldots,s_N}'$ is positive on $(r_{N+1},1)$ for conditions of case $\textbf{1}$, while equality $(2)$ shows that $h_{r_1,\ldots,r_{N+1},s_1,\ldots,s_N}'$ is positive on $(0,r_{N+1})$ for conditions of case \textbf{2}. Since $h_{r_1,\ldots,r_{N+1},s_1,\ldots,s_N}(r_{N+1})=0$ we get that $h_{r_1,\ldots,r_{N+1},s_1,\ldots,s_N}$ is positive on $(r_{N+1},1)$ for conditions of case \textbf{1}, while $h_{r_1,\ldots,r_{N+1},s_1,\ldots,s_N}$ is negative on $(0,r_{N+1})$ for conditions of case \textbf{2}. This ends the proof of monotonicity of $(u_N)_{N\geq 0}$ and conditions for strict monotonicity in both cases. Extending results from finite $N$ to $N\to \infty$ follows directly from these results and Remark~\ref{rem1}.

\subsection{Proof of Theorem~\ref{thm:main}}

We want to study the variations of 
\begin{equation*}
	x \mapsto f(x,a,b)=\left(\zeta(x,b)-\zeta(x,a+b)\right)^{\frac{1}{x}}
\end{equation*}
on $[1,\infty)$, for which it is enough, by continuity, to focus on $(1,\infty)$. 
Since $f$ is positive, its variations are equivalent to those of
\begin{align*}
	F(x,a,b)&\overset{\text{def}}{=}\ln f(x,a,b)\\
	&= \frac{1}{x}\ln \left(\zeta(x,b)-\zeta(x,a+b)\right)=\frac{1}{x}\ln\left( \sum_{k=0}^{\infty} \frac{1}{(k+b)^x}-\frac{1}{(k+a+b)^x}\right)\\
&=-\ln b+ \frac{1}{x} \ln\left( \sum_{k=0}^{\infty} \left(\frac{b}{k+b}\right)^x- \left(\frac{b}{k+a+b}\right)^x\right).
\end{align*}
A straightforward computation shows that
\begin{align*}
	\partial_x F(x,a,b)=
	\frac{H(x,a,b)}{x^2\left(\underset{k=0}{\overset{\infty}{\sum}} \left(\frac{b}{k+b}\right)^x- \left(\frac{b}{k+a+b}\right)^x\right)},%\Big[\\
%&\sum_{k=0}^{\infty} \left(\frac{b}{k+b}\right)^x \ln\left(\left(\frac{b}{k+b}\right)^x\right) + \left(\frac{b}{k+b+a}\right)^x \ln\left(\left(\frac{b}{k+b+a}\right)^x\right)\\
%&-\left(\sum_{k=0}^{\infty} \left(\frac{b}{k+b}\right)^x- \left(\frac{b}{k+a+b}\right)^x\right)\ln\left(\sum_{k=0}^{\infty} \left(\frac{b}{k+b}\right)^x- \left(\frac{b}{k+a+b}\right)^x\right)
%\Big].
\end{align*}
hence the sign of the derivative $\partial_x F(x,a,b)$ is the same as that of $H(x,a,b)$ defined by
%Thus the sign of this derivative is that of $H(x,a,b)$ defined by
% \sloppy{
\small{
\begin{align*}
&H(x,a,b)\overset{\text{def}}{=}\sum_{k=0}^{\infty} \left(\frac{b}{k+b}\right)^x \ln\left(\left(\frac{b}{k+b}\right)^x\right) - \left(\frac{b}{k+b+a}\right)^x \ln\left(\left(\frac{b}{k+b+a}\right)^x\right)\\
&\quad-\left(\sum_{k=0}^{\infty} \left(\frac{b}{k+b}\right)^x- \left(\frac{b}{k+a+b}\right)^x\right)\ln\left(\sum_{k=0}^{\infty} \left(\frac{b}{k+b}\right)^x- \left(\frac{b}{k+a+b}\right)^x\right)\\
&=\sum_{k=1}^{\infty}  \left(\frac{b}{k+b}\right)^x \ln\left(\left(\frac{b}{k+b}\right)^x\right) -\left(\frac{b}{k+b+a-1}\right)^x \ln\left(\left(\frac{b}{k+b+a-1}\right)^x\right)\\
&-\left(1+\sum_{k=1}^{\infty} \left(\frac{b}{k+b}\right)^x- \left(\frac{b}{k+a-1+b}\right)^x \right)\ln\left(1+\sum_{k=1}^{\infty} \left(\frac{b}{k+b}\right)^x- \left(\frac{b}{k+a-1+b}\right)^x \right),
\end{align*}
}
% }
%$H(x,a,b)$
\normalsize{which can be rewritten as}
\begin{align*}
	\sum_{n=1}^{\infty}  (s_n \ln s_n-r_n\ln r_n) -
	\left(1+\sum_{n=1}^{\infty} (s_n- r_n) \right)\ln\left(1+\sum_{n=1}^{\infty} (s_n-r_n) \right),
\end{align*}
where, for all $n \geq 1$, we have defined
\begin{equation*}
	s_n=\left(\frac{b}{n+b}\right)^x \quad \text{and} \quad r_n= \left(\frac{b}{n+a-1+b}\right)^x.
\end{equation*} 
Note that, for any values of $a>0$ and $b>0$, we have $\underset{n=1}{\overset{\infty}{\sum}} |s_n-r_n|<\infty$ as $x>1$ implies that $\underset{n=1}{\overset{\infty}{\sum}} s_n<\infty$ and $\underset{n=1}{\overset{\infty}{\sum}} r_n<\infty$. Moreover, when $a>1$, we have $0< r_n<s_n< 1$, while when $0<a<1$, we have $0<r_{n+1}<s_n<r_n<1$ for all $n\geq 1$ and $x>1$. 
We can then apply Lemma~\ref{lem:a-larger-1} to obtain that $H(x,a,b)$, and thus $\partial_x F(x,a,b)$, is negative if $a>1$ (case \textbf{1} of  Lemma~\ref{lem:a-larger-1}) and is positive if $0<a<1$ (case \textbf{2} of  Lemma~\ref{lem:a-larger-1}), which concludes the proof.

\section{Probabilistic interpretation: application to the beta-exponential distribution\label{sec:application}}

The aim of this section is to identify a connection between Theorem~\ref{thm:main} and the beta-exponential distribution. The latter arises by taking a log-transformation of a beta random variable. More specifically, let $V$ be a beta random variable with parameters $a>0$ and $b>0$, then we say that $X$ is an beta-exponential random variable with parameters $a$ and $b$ if $X = -\ln(1-V)$, and use the notation $X\sim \text{BE}(a,b)$. A three-parameter generalization of the beta-exponential distribution is studied in \cite{nadarajah2006beta}; a related family of distributions, named generalized exponential, is investigated in \cite{gupta1999theory}. The density function of $X\sim \text{BE}(a,b)$ is given by %cumulative distribution function, density function and hazard function are given respectively by
\begin{equation}\label{eq:dens}
    g(x;a,b) = \frac{1}{B(a,b)}(1-\edr^{-x})^{a-1}\edr^{-bx}\mathbf{1}_{(0,+\infty)}(x),
\end{equation}
where $B(a,b)$ denotes the beta function\footnote{The beta function is defined in this article as $B(a,b)=\int_0^{+\infty}(1-\edr^{-x})^{a-1}\edr^{-bx}\ddr x$ \citep[see eg][]{srivastava2012zeta}.
}. 
See the right panel of Figure~\ref{fig:proba-interpretation} for an illustration of densities $x\mapsto g(x;a,1)$ for different values of $a$.
The corresponding cumulant-generating function can be written as 
\begin{align*}
    K(t)%&=\ln \mathds{E}\left[\exp(t Y)\right]\\
    %&=\ln \mathds{E}\left[(1-V)^{-t}\right]\\
    &\overset{\text{def}}{=}\ln \E(\exp(tX))
    =\ln\Gamma(a+b)+\ln\Gamma(b-t)-\ln\Gamma(b)-\ln\Gamma(a+b-t),
\end{align*}
provided that $t<b$ \citep[see Section 3 of][]{nadarajah2006beta}. This implies that, for any $n\geq 1$, the $n^{\text{th}}$ cumulant of $X$, denoted $\kappa_n(a,b)$, is given by
\begin{equation}\label{eq:cum}
    %\forall\,\, n\geq 1\,:\,
    \kappa_n(a,b)%K^{(n)}(0)
    =(-1)^n\left(\psi^{(n-1)}(b)-\psi^{(n-1)}(b+a)\right).
\end{equation}
%where $\psi^{(m)}$ for $m\in\mathbb{N}$ denotes the polygamma function of order $m$, which is defined as the derivative of order $m+1$ of the logarithm of the gamma function. 
An interesting relation across cumulants of different orders is then obtained as a straightforward application of Theorem~\ref{thm:main}. Before stating the result, and for the sake of compactness, we define on $\mathbb{N}\setminus\{0\}$, for any $a>0$ and $b>0$, the function 
\begin{equation}\label{eq:cumulant-eb}
n\mapsto f_{\text{BE}}(n,a,b)=\left(\frac{\kappa_n(a,b)}{(n-1)!}\right)^{\frac{1}{n}}. %\quad \text{over} \quad (\mathbb{N}\setminus\{0\}) \times (0,\infty)\times (0,\infty).
\end{equation}

\begin{corollary}\label{cor:EB}
For any $b>0$, the function $n\mapsto f_{\mathrm{BE}}(n,a,b)$ %=\left(\kappa_n/\Gamma(n)\right)^{1/n}$, 
defined on $\mathbb{N}\setminus \{0\}$, is increasing if $0<a<1$, decreasing if $a>1$, and constantly equal to $\frac{1}{b}$ if $a=1$.
\end{corollary}

\begin{proof}
The proof follows by observing that, when $n\in\mathbb{N}\setminus\{0\}$, $f_{\text{BE}}(n,a,b)=f(n,a,b)$, %, that is $f_\text{BE}$ is the restriction to $\mathbb{N}\setminus\{0\}$ 
with the latter defined in
\eqref{eq:main} and \eqref{eq:main2}. This can be seen, when $n>1$, by applying twice the identity $\psi^{(n-1)}(s)=(-1)^{n} (n-1)! \zeta(n,s)$, and, when $n=1$, by applying twice the identity $\psi^{(0)}(s)=-\gamma+\underset{k=0}{\overset{\infty}{\sum}} \frac{s-1}{(k+1)(k+s)}$, where $\gamma$ is the Euler-Mascheroni constant, which holds for any $s>-1$ \citep[see Identity 6.3.16 in][]{Abr65}.
\end{proof}

As a by-product, a combination of Corollary~\ref{cor:EB} and \eqref{eq:cum} proves the chain of inequalities \eqref{eq:poly} presented in Section \ref{sec:main}.

%can alternatively be expressed as a chain of inequalities in terms of polygamma functions of different orders, which might be of independent interest. Namely, for any $b>0$ and any $0<a_1<1<a_2$, the following holds:
%\begin{equation}\label{eq:poly}
%    \begin{cases}\psi^{(0)}(b+a_1)-\psi^{(0)}(b)<\ldots<
%\left(\frac{\psi^{(n)}(b+a_1)-\psi^{(n)}(b)}{n!}\right)^{\frac{1}{n+1}}<\ldots <
%    \frac{1}{b},\\
%    \psi^{(0)}(b+a_2)-\psi^{(0)}(b)>\ldots>
%\left(\frac{\psi^{(n)}(b+a_2)-\psi^{(n)}(b)}{n!}\right)^{\frac{1}{n+1}}>\ldots >
%    \frac{1}{b}.
%     &<\ldots<\left(\frac{\psi^{(n)}(b+a_2)-\psi^{(n)}(b)}{(n-1)!}\right)^{1/n}
%   <\ldots<
%     \psi^{(0)}(b+a_2)-\psi^{(0)}(b).
%    \end{cases}
%\end{equation}
% can be rephrased by saying that, for any $b>0$, $f(n):=\left(\kappa_n/\Gamma(n)\right)^{1/n}$ is increasing in $n$ if $a<1$ and decreasing if $a>1$.\\
Corollary~\ref{cor:EB} %, as well as \eqref{eq:poly}, 
highlights the critical role played by the exponential distribution with mean $\frac{1}{b}$, special case of the beta-exponential distribution recovered from \eqref{eq:dens} by setting $a=1$. In such special instance, the cumulants simplify to $\kappa_n(1,b)=b^{-n}(n-1)!$, which makes $f_{\text{BE}}(n,1,b)=\frac{1}{b}$ for every $n\in\mathbb{N}\setminus\{0\}$. Within the beta-exponential distribution, the case $a=1$ then creates a dichotomy by identifying two subclasses of densities, namely $\{g(x;a,b)\,:\,0<a<1\}$, whose cumulants $\kappa_n(a,b)$ make $f_{\text{BE}}(n,a,b)$ an increasing function of $n$, and $\{g(x;a,b)\,:\,a>1\}$ for which $f_{\text{BE}}(n,a,b)$ is a decreasing function of $n$. The left panel of Figure~\ref{fig:proba-interpretation} is an illustration of Corollary~\ref{cor:EB}: it displays the function $b\mapsto f_{\text{BE}}(n,a,b)$ for values of $n\in\{1,2,3\}$ and $a\in [0.4,4]$, and it can be appreciated that for any $b$ in the considered range $[0,10]$ the order of the values taken by $f_{\text{BE}}(n,a,b)$ is in agreement with Corollary~\ref{cor:EB}.\\

\begin{figure}[ht]
    \includegraphics[width = .47\textwidth]{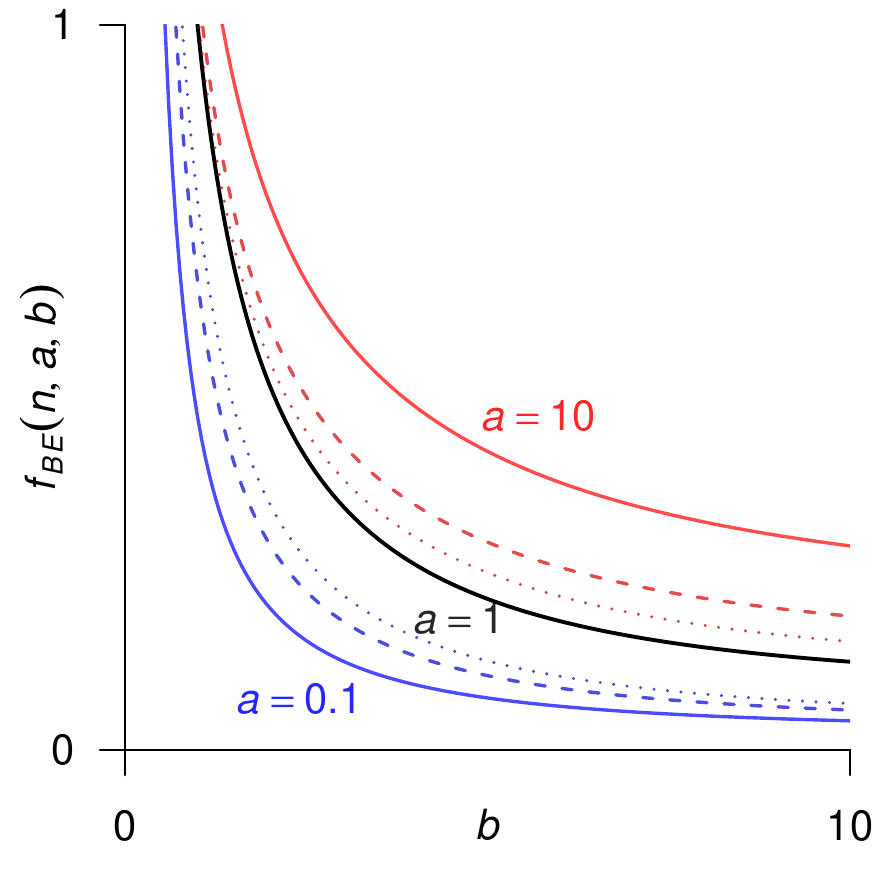} %[width = .47\textwidth]
    \includegraphics[width = .47\textwidth]{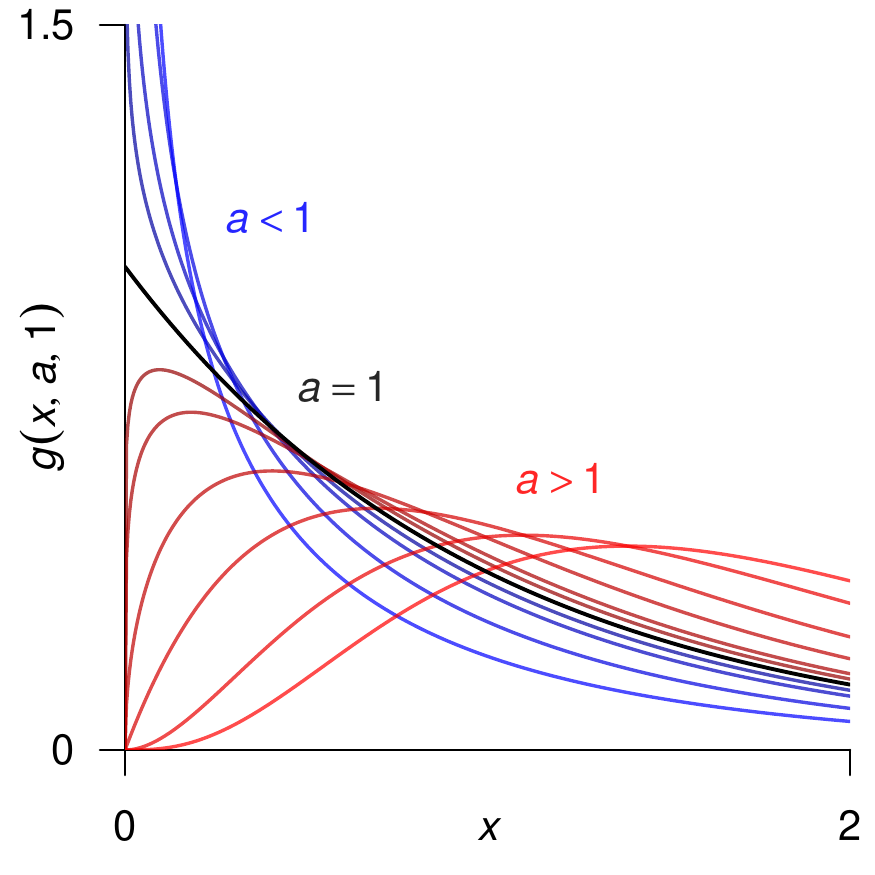} % [width = .47\textwidth]
  \caption{\csentence{Dichotomy of the beta-exponential distribution relative to the value $a=1$.}
      \textit{Left}: illustration of Corollary~\ref{cor:EB} displaying curves $b\mapsto f_{\text{BE}}(n,a,b)$, with $a$ taking values 0.1 (blue curves), 1 (black curve) and 10 (red curves), and $n$ taking values 1 (continuous curves), 2 (dashed curves) and 3 (dotted curves).
     \textit{Right}: illustration of Proposition~\ref{prop:shape-density} beta-exponential density function $g(x;a,1)$ for values of $a\in [0.4,4]$: densities are log-convex for $0<a<1$ (in blue), log-concave for $a>1$ (in red), while $a=1$ corresponds to the exponential distribution with mean 1 (in black).}
     \label{fig:proba-interpretation}
      \end{figure}
      
% \begin{figure}[ht]
%     \centering
%     \includegraphics[width = .47\textwidth]{f-functions-1.pdf}
%     \includegraphics[width = .47\textwidth]{densities-1.pdf}
%     \caption{
%     \textcolor{red}{Dichotomy of the beta-exponential distribution relative to the value $a=1$.} 
%     \textit{Left}: illustration of Corollary~\ref{cor:EB} displaying curves $b\mapsto f_{\text{BE}}(n,a,b)$, with $a$ taking values 0.1 (blue curves), 1 (black curve) and 10 (red curves), and $n$ taking values 1 (continuous curves), 2 (dashed curves) and 3 (dotted curves).
%      \textit{Right}: \textcolor{red}{illustration of Proposition~\ref{prop:shape-density}} beta-exponential density function $g(x;a,1)$ for values of $a\in [0.4,4]$: densities are log-convex for $0<a<1$ (in blue), log-concave for $a>1$ (in red), while $a=1$ corresponds to the exponential distribution with mean 1 (in black).}
%      \label{fig:proba-interpretation}
% \end{figure}

The first two cumulants of a random variable $X$ have a simple interpretation in terms of its first two moments, namely $\kappa_1 = \mathbb{E}[X]$ and $\kappa_2 = \text{Var}[X].$
%\begin{align*}
%    \kappa_1 = \mathbb{E}[X] \; \text{ and } \; \kappa_2 = \text{Var}[X].
%\end{align*}
%
A special case of Corollary~\ref{cor:EB}, focusing on the case $n\in\{1,2\}$, then provides an interesting result relating the dispersion of the beta-exponential distribution with its mean. Specifically,
\sloppy{\begin{corollary}\label{cor:EB-dispersion}
For any $b>0$, the beta-exponential random variable $X\sim\mathrm{BE}(a,b)$ is characterized by
over-dispersion $\left(\sqrt{\mathrm{Var}[X]}>\mathbb{E}[X]\right)$ if $0<a<1$,
under-dispersion $\left(\sqrt{\mathrm{Var}[X]}<\mathbb{E}[X]\right)$ if $a>1$, and
equi-dispersion $\left(\sqrt{\mathrm{Var}[X]}=\mathbb{E}[X]\right)$ if $a=1$.
\end{corollary}
}

%\begin{remark}
The behaviour of the cumulants is not the only distinctive feature characterizing the two subclasses of density functions corresponding to $0<a<1$ and $a>1$. For any $b$, the value of $a$ determines the shape of the density as displayed in the right panel of Figure~\ref{fig:proba-interpretation} and summarized by the next proposition, whose proof is trivial and thus omitted: for any $0<a<1$, the density is log-convex (curves in blue on the right panel of Figure~\ref{fig:proba-interpretation}), while for any $a>1$, the density is log-concave (curves in red on the right panel of Figure~\ref{fig:proba-interpretation}); the case $a=1$ corresponds to the exponential distribution (curve in black on the right panel of Figure~\ref{fig:proba-interpretation}). %easily follows by observing that the second derivative of $\ln g$ has the opposite sign as $a-1$.
% have the following dichotomy on the shape of the density depending on the parameter $a$ only.
\begin{proposition}\label{prop:shape-density}
    For any $b>0$, the beta-exponential density $g(x;a,b)$ is log-convex if $0<a<1$ and log-concave if $a>1$.
\end{proposition}
%\begin{proof} It follows by observing that the second derivative of $\ln g$ has the opposite sign of $a-1$.
    %\begin{equation*}
    %    \frac{\ddr^2}{\ddr x^2}\ln g(x) = -(a-1)\frac{\edr^{-x}}{(1-\edr^{-x})^2}.
    %\end{equation*}
%\end{proof}

The same dichotomy within the beta-exponential distribution is further highlighted by the behaviour of the corresponding hazard function, %, that is the event rate at time $x$ conditional on survival until time $x$ or later. 
%More specifically, the hazard 
defined for an absolutely continuous random variable $X$ as the function $x\mapsto \frac{f_X(x)}{1-F_X(x)}$, where $f_X$ and $F_X$ are, respectively, the probability density function and  the cumulative distribution function of $X$.

\begin{proposition}\label{prop:hazard}
    For any $b>0$, the  hazard function of the beta-exponential distribution with parameters $a$ and $b$ is decreasing if $a<1$, increasing if $a>1$, and constantly equal to $b$ if $a=1$.
\end{proposition}
\begin{proof}
    The result follows from the log-convexity and log-concavity properties of $g(x;a,b)$ \citep[see][]{barlow1975statistical}.\\
\end{proof}

Finally, it is worth remarking that an analogous dichotomy holds within the class of gamma density functions with $a>0$ and $b>0$ shape and rate parameters, and that once again the exponential distribution with mean $\frac{1}{b}$, special case recovered by setting $a=1$, lays at the border between the two subclasses. The $n^{\text{th}}$ cumulant of the gamma distribution is $\kappa_n=ab^{-n}(n-1)!$, which makes the function $n\mapsto \left(\frac{\kappa_n(a,b)}{(n-1)!}\right)^{\frac{1}{n}}$, defined on $\mathbb{N}\setminus \{0\}$, increasing if $0<a<1$, decreasing if $a>1$ and constantly equal to $\frac{1}{b}$ if $a=1$. Similarly, the gamma density is log-convex if $0<a<1$ and log-concave if $a>1$ and, thus, the corresponding hazard function is decreasing if $a<1$, increasing if $a>1$ and constantly equal to $b$ if $a=1$.
%\end{remark}

%%%%%%%%%%%%%%%%%%%%%%%%%%%%%%%%%%%%%%%%%%%%%%
%%                                          %%
%% Backmatter begins here                   %%
%%                                          %%
%%%%%%%%%%%%%%%%%%%%%%%%%%%%%%%%%%%%%%%%%%%%%%

\begin{backmatter}

\section*{Funding}

O.M. would like to thank Universit\'e Lyon $1$, Universit\'e Jean Monnet and Institut Camille Jordan for material support. This work was developed in the framework of 
the Ulysses  Program for French-Irish collaborations (43135ZK), 
the Grenoble Alpes Data Institute and 
the LABEX MILYON (ANR-10-LABX-0070) of Universit\'e de Lyon, within 
the program ``Investissements d'Avenir'' (ANR-11-IDEX-0007) and (ANR-15-IDEX-02) operated by the French National Research Agency (ANR).

\section*{Acknowledgements}

The authors would like to thank \href{https://webusers.imj-prg.fr/marco.mazzola}{Marco Mazzola} for fruitful suggestions. %us with an early proof of part of Theorem~\ref{thm:main}. 

\bibliographystyle{bmc-mathphys} % Style BST file (bmc-mathphys, vancouver, spbasic).
\bibliography{biblio}      % Bibliography file (usually '*.bib' )
\end{backmatter}
\end{document}